\documentclass[11pt, a4paper]{article}
\usepackage{amsfonts, amsmath, amssymb, amsthm, longtable, url, enumitem}

\newcommand\al{\alpha}
\newcommand\lm{\lambda}
\newcommand\sg{\sigma}

\newcommand\cF{\mathcal{F}}

\theoremstyle{plain}
\newtheorem{theorem}{Theorem}[section]

\newtheorem{corollary}[theorem]{Corollary}
\newtheorem{lemma}[theorem]{Lemma}
\newtheorem{proposition}[theorem]{Proposition}
\newtheorem*{claim*}{Claim}
\newtheorem*{problem*}{Problem}
\newtheorem*{conjecture*}{Conjecture}

\theoremstyle{definition}
\newtheorem{definition}[theorem]{Definition}

\newcommand\la{\langle}
\newcommand\ra{\rangle}
\newcommand\lla{\langle\!\langle}
\newcommand\rra{\rangle\!\rangle}

\newcommand\ad{\mathrm{ad}}

\newcommand{\gr}{\mathrm{gr}}

\newcommand\Z{\mathbb{Z}}

\newcommand\F{\mathbb{F}}

\newcommand\Aut{\mathrm{Aut}}
\newcommand\Out{\mathrm{Out}}

\newcommand{\A}{\mathrm{A}}
\newcommand{\B}{\mathrm{B}}
\newcommand{\C}{\mathrm{C}}
\renewcommand{\L}{\mathrm{L}}

\renewcommand{\phi}{\varphi}
\renewcommand{\epsilon}{\varepsilon}

\usepackage{ltablex} % So tabularx tables can break over pages
\usepackage{tabularx}
\keepXColumns

\setlist[enumerate,1]{label={\upshape\arabic*.}}
\setlist[enumerate,2]{label={\textup{(}\alph*\textup{)}}}

\newcolumntype{C}[1]{>{\centering\arraybackslash}m{#1}}
\newcolumntype{Y}{>{\centering\arraybackslash}X}

\title{Enumerating $3$-generated axial algebras of Monster type}
\author{S.M.S.~Khasraw\footnote{Department of Mathematics, College of Education, Salahaddin University-Erbil, Erbil, Kurdistan Region, Iraq, email: sanhan.khasraw@su.edu.krd}
\and
 J.~M\textsuperscript{c}Inroy\footnote{School of Mathematics, University of Bristol, Fry Building, Woodland Road, Bristol, BS8 1UG, UK, and the Heilbronn Institute for Mathematical Research, Bristol, UK, email: justin.mcinroy@bristol.ac.uk}
 \and
 S.~Shpectorov\footnote{School of Mathematics, University of Birmingham, Edgbaston, Birmingham, B15 2TT, UK, email: S.Shpectorov@bham.ac.uk}}

\date{}

\begin{document}
\maketitle

\begin{abstract}
An axial algebra is a commutative non-associative algebra generated by axes, 
that is, primitive, semisimple idempotents whose eigenvectors multiply 
according to a certain fusion law.  The Griess algebra, whose automorphism 
group is the Monster, is an example of an axial algebra.  We say an axial 
algebra is of \emph{Monster type} if it has the same fusion law as the 
Griess algebra.

The $2$-generated axial algebras of Monster type, called Norton-Sakuma algebras, have been fully classified and are one of nine isomorphism types.  In this paper, 
we enumerate a subclass of $3$-generated axial algebras of Monster type in terms of their groups and shapes.  It turns out that the vast majority of the possible shapes for such algebras collapse; that is they do not lead to non-trivial examples.  This is in sharp contrast to previous thinking.  Accordingly, we develop a method of minimal forbidden configurations, to allow us to efficiently recognise and eliminate collapsing shapes.
\end{abstract}

\section{Introduction}

The paradigm of axial algebras was introduced by Hall, Rehren and Shpectorov in \cite{Axial1, Axial2} and is a different way of describing classes of non-associative algebras.  Namely, the algebras are axiomatised in terms of special idempotent elements, called axes, whose adjoint action is described by a specific fusion law.  It turns out that such objects occur in different parts of mathematics and even beyond.  The origin of this approach goes back to quantum physics, where the axial properties are exhibited within vertex operator algebras.  Among the other connections, Tkachev \cite{Tkachev} noticed that Hsiang algebras arising in the theory of non-linear PDEs are axial algebras and he determined their fusion law.  Recently, a reference to axial algebras appeared in a paper \cite{Fox} studying algebras of vector fields on manifolds, such as the algebra of Ricci flows.

The notable examples of axial algebras include the Griess algebra and also a majority of simple Jordan algebras.  All these are contained in the class of axial algebras of Monster type $(\alpha, \beta)$.  The Griess algebra arises for $(\alpha, \beta) = (\frac{1}{4}, \frac{1}{32})$ and this is, in fact, the most exciting case, where many interesting examples occur.  In the remainder of the paper, where we talk about algebras of Monster type, we mean Monster type $(\frac{1}{4}, \frac{1}{32})$.

The fusion law which axiomatises the class of algebras of Monster type assures that for every axis $a$ there is an associated automorphism $\tau_a$ of the algebra, called the \emph{Miyamoto involution}.  The group generated by all Miyamoto involutions is known as the \emph{Miyamoto group} of the algebra.  For the Griess algebra, the Miyamoto group is the Monster sporadic simple group and this example was the one that motivated the whole theory of axial algebras, indirectly via the theory of Majorana algebras \cite{Majoranabook, majsurvey}.

Apart from Sakuma's theorem \cite{IPSS, Axial1}, which classifies $2$-generated algebras of Monster type, most other papers study algebras for specific Miyamoto groups.  We want to be more systematic and enumerate a class of algebras defined by some natural conditions.  The theory is still in its infancy and so we cannot be too ambitious, but it looks natural to focus on the next case after Sakuma's theorem, namely the class of $3$-generated axial algebras of Monster type.  However, even for this class there are significant obstacles.

\bigskip\noindent{\bf Transposition groups and algebras.} The conditions that we impose on the algebra should be of the sort that can translate to group theoretic conditions, allowing us to find the Miyamoto groups as a first step.  It follows from Sakuma's theorem, that every Miyamoto group coming from an algebra of Monster type is a $6$-transposition group (see Corollary \ref{6trans}).  That is, the conjugacy class of Miyamoto involutions satisfies the property that the order of $xy$ is at most $6$ for any two Miyamoto involutions $x$ and $y$.  (Indeed, it is very well-known that the Monster is a $6$-transposition group.)  It would be very interesting to find the complete list of all $3$-generated $6$-transposition groups.  In fact, Bernd Fischer at conferences often stressed this as a key problem.  However, the complete list would involve the Monster as well as a large number of its subgroups, so this is a very difficult problem and is likely to be out of reach at present.  On the other hand, the class of $3$-transposition groups is very well understood and we have a complete list of $3$-generated $3$-transposition groups.  In this paper, we look at the class of $3$-generated $4$-algebras, that is the algebras of Monster type with no subalgebras of type $5\A$, or $6\A$.  This corresponds to their Miyamoto group being a $4$-transposition group.  By setting the limit at $4$, we also avoid another severe difficulty, namely the group $B(2,5){:}2$, which is a $3$-generated $5$-transposition group.  This is currently an open case of the Burnside problem, i.e.\ it is not known whether or not $B(2,5)$ is finite.

This project, finding $3$-generated $4$-algebras, was attempted in \cite{sanhan} by the first author.  This was incomplete in a number of ways, particularly because only a much weaker GAP realisation of the expansion algorithm \cite{axialconstruction} was available at the time.  However, even in that preliminary attempt, one could observe the main phenomenon that we emphasise here.  That the great majority of all shapes (configurations of $2$-generated subalgebras, as introduced in Section \ref{sec:shapes}) collapse, i.e.\ they do not lead to any non-trivial algebras.  This is in sharp contrast with the earlier intuition, where we would typically expect exactly one example in each case.  This can be explained partly as the early attempts dealt with simple groups, whereas in the current project most groups are soluble.  Even taking this into account, we were not prepared for the outcome.  Over $99\%$ of all cases (all but at most $101$ out of $11,257$) collapsed in our systematic search.  This shifts the focus of the theory from just building examples to the following questions.

\begin{problem*}
\begin{enumerate}
\item What theoretical conditions can we impose on the shapes to ensure that they lead to non-trivial algebras?
\item Dually, without any extra conditions, how can we practically eliminate the majority of collapsing shapes?
\end{enumerate}
\end{problem*}

For the first question, we do not have any good suggestions at present and we note that the extra assumptions that have been used in Majorana theory are not entirely satisfactory as they eliminate good shapes, leading to valid algebras, as well as the bad ones.  

For the dual question, we argue in this paper that we should study minimal collapsing configurations.   We find experimentally $18$ such forbidden configurations, see Table \ref{tab:min0dims},  which eliminate almost all of our bad shapes.  The idea here is simple, if we have a small collapsing configuration which we can find as a subshape inside a larger shape, then this larger shape must also collapse.  Finding the subshape is much cheaper as it involves very simple group theory and combinatorics, whereas collapsing the shape by expansions involves finding exact solutions to huge systems of linear equations.

Using this idea and the $18$ forbidden configurations eliminates all but $123$ shapes.  On these we run the {\sc magma} implementation of the expansion algorithm which collapses a further $22$ shapes and we show that $45$ more complete to give non-trivial algebras.  The remaining $56$ shapes could not be completed by the algorithm.  We do not think that a simple improvement in performance of our algorithm can resolve these cases.  In particular, for at least some of the shapes, there are infinitely many different algebras that arise and so they cannot be computed by the present algorithm.  For example, the group $G=2^2$ on $2+2+2$ axes of shape $4\A$ was analysed by Whybrow \cite{maddyinfinite} and she showed that there are infinitely many algebras in this case.  Peacock and the second author \cite{axialvarieties} analysed the group $G=2^2$ on $2+2+2$ axes of shape $4\A(2\A)^2$ and while the calculation wasn't completed in all cases, it appears that there is more than one algebra in this case also.  So we expect that many of the remaining cases truly require new ideas, regardless of the size of the group.

Some of the completed algebras have appeared elsewhere before.  For example, four of the $S_4$ algebras first appeared in \cite{IPSS}, two of the $PSL_2(7)$ cases were done by hand in \cite{L32}, and the remainder of the shapes for these two groups appeared in \cite{axialconstruction}.  Also a couple of the results for soluble groups appeared there too, but as examples of algebras which are not $2$-closed and they actually originated in this project.  There is also some intersection with the tables from \cite{min3gen}.  There they investigate minimal $3$-generated algebras of Monster type.  Even though they allow $k=5,6$, their assumptions only lead to small Miyamoto groups.  Furthermore, the minimality assumption rules out many possible actions on axes and many shapes.  The second author \cite{3gen}, only assuming minimality of the group rather than the algebra, encountered a total of just $161$ shapes.

As we already mentioned, the thesis \cite{sanhan} was the first attempt at the classification of $3$-generated $4$-algebras.  However, there are significant differences.  The larger $2$-groups were not attempted.  More importantly, the approach taken in \cite{sanhan} followed the idea of Majorana representations of Ivanov, whereby the action on axes comes from the action on involutions in a central cover of the Miyamoto group.  On the one hand, this eliminates some of the actions, but on the other, this also causes some algebras to appear twice because the same action may appear in different covering groups.  Still, having a much weaker algorithmic support, \cite{sanhan} managed to find quite a few algebras and, most importantly, first noticed the collapsing phenomenon.  It also contained handmade proofs both for some algebras which exist and also for others which collapse.

The structure of the paper is as follows.  In Section \ref{sec:background} we give brief details about axial algebras and the shape of an algebra.  The group theory part of the project, finding $3$-generated $4$-transposition groups up to similarity, is completed in Section \ref{sec:groups}.  In Section \ref{sec:axes}, we develop general techniques for determining the possible actions on axes and we apply this to our case.  We compute all possible shapes in Section \ref{sec:forbidden} and face the problem that most of these need to be eliminated without running the expansion algorithm.  Hence we introduce the idea of forbidden configurations and give our list of $18$.  We also provide handmade proofs for two of these minimal collapsing shapes, which between them eliminate over half of our $11,257$ shapes.  Finally, in Section \ref{sec:results}, we run the algorithm on the shapes which cannot be eliminated via forbidden configurations.  We include two tables, one summarising the number of collapsing shapes per action and another listing the non-collapsing shapes individually.

%%%%%%%%%%%%%%%%%%%%%%%%%%%%%%%%%%%%%%

\section{Background}\label{sec:background}

We will review the definition and some properties of axial algebras which were 
first introduced by Hall, Rehren and Shpectorov in \cite{Axial1, Axial2} (see also \cite{DPSV} for a more modern version).

\begin{definition}
A \emph{fusion law} is a pair $\cF := (\cF, \star)$ consisting of a non-empty set 
$\cF$ and a binary operation $\star \colon \cF \times \cF \to 2^\cF$.
\end{definition}

Here $2^\cF$ denotes the set of all subsets of $\cF$.

The following Table \ref{tab:monsterfusion} is an example of a fusion law and in 
fact this is the fusion law we will be interesting in in this paper.  In the $\lm, 
\mu$ entry of the table, we list the elements of the set $\lm \star \mu$.  In 
particular, if the entry is empty, then $\lm \star \mu = \emptyset$.

\begin{table}[!htb]
\setlength{\tabcolsep}{4pt}
\renewcommand{\arraystretch}{1.5}
\centering
\begin{tabular}{c||c|c|c|c}
 & $1$ & $0$ & $\frac{1}{4}$ & $\frac{1}{32}$ \\ \hline \hline
$1$ & $1$ &  & $\frac{1}{4}$ & $\frac{1}{32}$ \\ \hline
$0$ &  & $0$ &$\frac{1}{4}$ & $\frac{1}{32}$ \\ \hline
$\frac{1}{4}$ & $\frac{1}{4}$ & $\frac{1}{4}$ & $1, 0$ & $\frac{1}{32}$ \\ \hline
$\frac{1}{32}$ & $\frac{1}{32}$  & $\frac{1}{32}$ & $\frac{1}{32}$ & $1, 0, \frac{1}{4}$ 
\end{tabular}
\caption{Monster fusion law}\label{tab:monsterfusion}
\end{table}

Suppose that $A$ is a non-associative algebra over a field $\F$.  We denote the 
adjoint of $a \in A$ by $\ad_a$ which is the map $x \mapsto ax$.  For $\lm \in \F$, 
we denote the $\lm$-eigenspace of $\ad_a$ by $A_\lm(a)$ and we will write $A_\lm$ 
where $a$ is understood.  We also write $A_S(a) = \bigoplus_{\lm \in S} A_\lm(a)$, 
where $S \subseteq \F$.

\begin{definition}
Suppose that $A$ is a commutative non-associative algebra over a field $\F$ and 
$\cF$ is a symmetric fusion law with $\cF \subset \F$.  The pair $A = (A, X)$ is an 
\emph{$\cF$-axial algebra} if $X \subseteq A$ is a generating set of non-zero 
idempotents and every $a \in X$ satisfies the following:
\begin{enumerate}
\item $A = A_\cF(a)$;  that is, $\ad_a$ is semisimple with all eigenvalues in $\cF$;
\item $A_\lm(a)A_\mu(a) \subseteq A_{\lm \star \mu}(a)$ for all $\lm,\mu\in\cF$.
\end{enumerate}
\end{definition}

Elements of the set $X$ are called \emph{axes}.  We will also drop the $\cF$- where 
it is understood and just talk of axial algebras.

Since each $a \in X$ is an idempotent, one of the eigenvalues in $\cF$ is $1$.  We 
call an axis $a$ \emph{primitive} if $A_1(a) = \la a \ra$ is $1$-dimensional.  We 
say that $A$ is \emph{primitive} if every axis $a \in X$ is primitive.

The Griess algebra, used to construct the Monster sporadic simple group, is an 
example of an axial algebra with the fusion law in Table \ref{tab:monsterfusion}.  
Hence we call this fusion law the \emph{Monster fusion law}.  If an axial algebra 
has the Monster fusion law, we say that it is of \emph{Monster type}.  Recently, 
there has been a lot of interest in such algebras and also those with a generalised 
form of the Monster fusion law, with arbitrary $\al$ and $\beta$ in place of 
$\frac{1}{4}$ and $\frac{1}{32}$.

The Griess algebra also has a bilinear form which associates with the algebra product.

\begin{definition}
A Frobenius form for an axial algebra $A$ is a non-zero bilinear form $(\cdot,\cdot) 
\colon A \times A \to \F$ such that for all $x,y,z \in A$
\[
(x,yz) = (xy,z).
\]
\end{definition}

Note that a Frobenius form is automatically symmetric \cite[Proposition 3.5]{Axial1}.  
In the Griess algebra, the Frobenius form also has the property that $(a,a)=1$ for 
all axes $a$, but we don't require this in our definition.

\subsection{Gradings and automorphisms}

The key property that axial algebras and Majorana algebras generalise from the Griess 
algebra is that there is a natural link between axes and involutory automorphisms.  
This link occurs precisely when we have a graded fusion law.

In general, axial algebras can be graded by any group $T$, but the axial algebras of 
Monster type, which we are particularly concerned with in this paper, have a 
$\Z_2$-grading.  So here we give a simplified version of the definition of a 
$\Z_2$-grading.  For the more general $T$-grading see \cite{axialstructure} and for 
a categorical treatment see \cite{DPSV}.  We will write $\Z_2$ as $\{ +, - \}$ with 
the usual multiplication of signs.

\begin{definition}
A \emph{$\Z_2$-grading} on a fusion law $\cF$ is a map $\gr \colon \cF \to \Z_2$ such that $\gr(\nu) = \gr(\lambda) \gr(\mu)$ for all $\lambda, \mu \in \cF$ and all $\nu \in \lambda \star \mu$.
\end{definition}

Note that the Monster fusion law $\mathcal{M}$ is $\mathbb{Z}_2$-graded where $1,0,\frac{1}{4}$ all map to $+$ and $\frac{1}{32}$ maps to $-$.  We will write $A_\epsilon := \bigoplus_{\lambda \in \gr^{-1}(\epsilon)} A_\lambda$ for $\epsilon = \pm$.  Let $a$ be an axis in $A$.  We define a map $\tau_a \colon A \to A$ by
\[
v \mapsto \gr(\lambda) v
\]
for $v \in A_\lambda$, $\lambda \in \cF$, and extend linearly to $A$.  Since multiplication obeys the fusion law, this is an automorphism of $A$ and we call $\tau_a$ the \emph{Miyamoto involution} associated to $a$.  Note that $\tau_a$ negates $A_-$ and acts as the identity on $A_+$.  Even though we call $\tau_a$ an involution, it is the identity if $A_- = 0$.  So that we can have involutions coming from the grading, we will always assume that $\F$ does not have characteristic $2$, because in this case $-1=+1$.

The following is an easy lemma.

\begin{lemma}\label{autaxis}
Suppose that $A$ is an $\mathcal{F}$-axial algebra.  Let $a \in X$ and $g \in \Aut(A)$.  Then $a^g$ is also an axis of $A$ and
\[
A_\lambda(a^g) = A_\lambda(a)^g
\]
for all $\lambda \in \mathcal{F}$.  If $\mathcal{F}$ is $\mathbb{Z}_2$-graded, then $\tau_a^g = \tau_{a^g}$.
\end{lemma}

Since we have a (possibly) different automorphism for each axis $a$, this generates a group of automorphisms.

\begin{definition}
Let $A$ be an axial algebra with a $\mathbb{Z}_2$-graded fusion law.  Then the \emph{Miyamoto group} is the group
\[
G(X) := \langle \tau_a : a \in X \rangle.
\]
\end{definition}

We may also generalise this notation and consider the group $G(Y) := \langle \tau_a : a \in Y \rangle$ for any subset $Y \subseteq X$.  We define $\bar{Y} = Y^{G(Y)}$.  It turns out that $G(\bar{Y}) = G(Y)$ and so $\bar{Y}^{G(\bar{Y})} = \bar{Y}$.  We call $\bar{Y}$ the \emph{closure} of $Y$ and we say that $Y$ is \emph{closed} if $Y = \bar{Y}$.  In \cite{axialstructure}, it is also shown that $\lla Y \rra = \lla \bar{Y} \rra$, where $\lla Y \rra$ is the algebra generated by $Y$.  In this paper, we will normally assume that the set $X$ of axes is closed as we can always enlarge $X$ to $\bar{X}$ without changing the algebra, or Miyamoto group.  When the set $X$ is closed, $G = G(X)$ acts faithfully on $X$.  

In this project, we will be classifying a large number of algebras $(A, X)$ in terms of the action of $G(X)$ on $X$ and also their $\tau$-map from $X$ to $G(X)$.  In addition, we will also introduce something called the shape of the algebra and this depends on our knowledge of the $2$-generated algebras.

%%%%%%%%%%%%%%%%%%

\subsection{Norton-Sakuma algebras}\label{sec:dihedral}
\begin{table}[p]
\setlength{\tabcolsep}{4pt}
\renewcommand{\arraystretch}{1.5}
\centering
\footnotesize
\begin{tabular}{c|c|c}
Type & Basis & Products \& form \\ \hline
$2\textrm{A}$ & \begin{tabular}[t]{c} $a_0$, $a_1$, \\ $a_\rho$ 
\end{tabular} & 
\begin{tabular}[t]{c}
$a_0 \cdot a_1 = \frac{1}{2^3}(a_0 + a_1 - a_\rho)$ \\
$a_0 \cdot a_\rho = \frac{1}{2^3}(a_0 + a_\rho - a_1)$ \\
$(a_0, a_1) = (a_0, a_\rho)= (a_1, a_\rho) = \frac{1}{2^3}$
\vspace{4pt}
\end{tabular}
\\
$2\textrm{B}$ & $a_0$, $a_1$ &
\begin{tabular}[t]{c}
$a_0 \cdot a_1 = 0$ \\
$(a_0, a_1) = 0$
\vspace{4pt}
\end{tabular}
\\
$3\textrm{A}$ & \begin{tabular}[t]{c} $a_{-1}$, $a_0$, \\ $a_1$, $u_\rho$ \end{tabular} &
\begin{tabular}[t]{c}
$a_0 \cdot a_1 = \frac{1}{2^5}(2a_0 + 2a_1 + a_{-1}) - \frac{3^3\cdot5}{2^{11}} u_\rho$ \\
$a_0 \cdot u_\rho = \frac{1}{3^2}(2a_0 - a_1 - a_{-1}) + \frac{5}{2^{5}} u_\rho$ \\
$u_\rho \cdot u_\rho = u_\rho$, $(a_0, a_1) = \frac{13}{2^8}$ \\
$(a_0, u_\rho) = \frac{1}{2^2}$, $(u_\rho, u_\rho) = \frac{2^3}{5}$ 
\vspace{4pt}
\end{tabular}
\\
$3\textrm{C}$ & \begin{tabular}[t]{c} $a_{-1}$, $a_0$, \\ $a_1$ \end{tabular} &
\begin{tabular}[t]{c}
$a_0 \cdot a_1 = \frac{1}{2^6}(a_0 + a_1 - a_{-1})$ \\
$(a_0, a_1) = \frac{1}{2^6}$
\vspace{4pt}
\end{tabular}
\\
$4\textrm{A}$ & \begin{tabular}[t]{c} $a_{-1}$, $a_0$, \\ $a_1$, $a_2$ \\ $v_\rho$ \end{tabular} &
\begin{tabular}[t]{c}
$a_0 \cdot a_1 = \frac{1}{2^6}(3a_0 + 3a_1 + a_{-1} + a_2 - 3v_\rho)$ \\
$a_0 \cdot v_\rho = \frac{1}{2^4}(5a_0 - 2a_1 - a_2 - 2a_{-1} + 3v_\rho)$ \\
$v_\rho \cdot v_\rho = v_\rho$, $a_0 \cdot a_2 = 0$ \\
$(a_0, a_1) = \frac{1}{2^5}$, $(a_0, a_2) = 0$\\
$(a_0, v_\rho) = \frac{3}{2^3}$, $(v_\rho, v_\rho) = 2$
\vspace{4pt}
\end{tabular}
\\
$4\textrm{B}$ & \begin{tabular}[t]{c} $a_{-1}$, $a_0$, \\ $a_1$, $a_2$ \\ $a_{\rho^2}$ \end{tabular} &
\begin{tabular}[t]{c}
$a_0 \cdot a_1 = \frac{1}{2^6}(a_0 + a_1 - a_{-1} - a_2 + a_{\rho^2})$ \\
$a_0 \cdot a_2 = \frac{1}{2^3}(a_0 + a_2 - a_{\rho^2})$ \\
$(a_0, a_1) = \frac{1}{2^6}$, $(a_0, a_2) = (a_0, a_{\rho^2})= \frac{1}{2^3}$
\vspace{4pt}
\end{tabular}
\\
$5\textrm{A}$ & \begin{tabular}[t]{c} $a_{-2}$, $a_{-1}$,\\ $a_0$, $a_1$,\\ $a_2$, $w_\rho$ \end{tabular} &
\begin{tabular}[t]{c}
$a_0 \cdot a_1 = \frac{1}{2^7}(3a_0 + 3a_1 - a_2 - a_{-1} - a_{-2}) + w_\rho$ \\
$a_0 \cdot a_2 = \frac{1}{2^7}(3a_0 + 3a_2 - a_1 - a_{-1} - a_{-2}) - w_\rho$ \\
$a_0 \cdot w_\rho = \frac{7}{2^{12}}(a_1 + a_{-1} - a_2 - a_{-2}) + \frac{7}{2^5}w_\rho$ \\
$w_\rho \cdot w_\rho = \frac{5^2\cdot7}{2^{19}}(a_{-2} + a_{-1} + a_0 + a_1 + a_2)$ \\
$(a_0, a_1) = \frac{3}{2^7}$, $(a_0, w_\rho) = 0$, $(w_\rho, w_\rho) = \frac{5^3\cdot7}{2^{19}}$
\vspace{4pt}
\end{tabular}
\\
$6\textrm{A}$ & \begin{tabular}[t]{c} $a_{-2}$, $a_{-1}$,\\ $a_0$, $a_1$,\\ $a_2$, $a_3$ \\ $a_{\rho^3}$, $u_{\rho^2}$ \end{tabular} &
\begin{tabular}[t]{c}
$a_0 \cdot a_1 = \frac{1}{2^6}(a_0 + a_1 - a_{-2} - a_{-1} - a_2 - a_3 + a_{\rho^3}) + \frac{3^2\cdot5}{2^{11}}u_{\rho^2}$ \\
$a_0 \cdot a_2 = \frac{1}{2^5}(2a_0 + 2a_2 + a_{-2}) - \frac{3^3\cdot5}{2^{11}}u_{\rho^2}$ \\
$a_0 \cdot u_{\rho^2} = \frac{1}{3^2}(2a_0 - a_2 - a_{-2}) + \frac{5}{2^5}u_{\rho^2}$ \\
$a_0 \cdot a_3 = \frac{1}{2^3}(a_0 + a_3 - a_{\rho^3})$, $a_{\rho^3} \cdot u_{\rho^2} = 0$\\
$(a_0, a_1) = \frac{5}{2^8}$, $(a_0, a_2) = \frac{13}{2^8}$ \\
$(a_0, a_3) = \frac{1}{2^3}$, $(a_{\rho^3}, u_{\rho^2}) = 0$,
\end{tabular}
\end{tabular}
\caption{Norton-Sakuma algebras}\label{tab:sakuma}
\end{table}

The $2$-generated axial algebras of Monster type have been classified in a number of papers under slightly different assumptions and the most general result is the following.

\begin{theorem}\textup{\cite{2gen}}\footnote{This result was first announced at the Axial Algebra Focused Workshop in Bristol in May 2018.  It has also been checked computationally by M. Whybrow.}
A primitive $2$-generated axial algebra of Monster type over a field of characteristic $0$ is one of the eight listed in Table $\ref{tab:sakuma}$.  They are called \emph{Norton-Sakuma} algebras.
\end{theorem}

We will now explain how to interpret Table \ref{tab:sakuma}.  Let $A = \lla a_0, a_1 \rra$ be one of the Norton-Sakuma algebras.  Since $A$ is $2$-generated, the Miyamoto group $G = \la \tau_{a_0}, \tau_{a_1} \ra \cong D_{2m}$ is a dihedral group.  Let $\rho := \tau_{a_0}\tau_{a_1}$.  We define
\[
a_{\epsilon + 2k} = a_\epsilon^{\rho^k}
\]
By Lemma \ref{autaxis}, $a_i$ is also an axis of $A$ and so the set $X$ of all axes $a_i$ is a closed set of axes.  The algebras are all named $n\L$, where $n = |X|$ and $\L$ is some letter.  For most of the algebras, the axes $a_i$ do not span the algebra.  For these algebras, we introduce additional elements which are indexed by powers of $\rho$.  The Miyamoto group $G$ fixes each of these additional elements apart from $w_\rho$ in $5\A$ which is inverted by each involution.  We only give the multiplication and values in the Frobenius form for some of the basis elements.  The remaining ones can be deduced by using the action of the Miyamoto group.

From Table \ref{tab:sakuma}, we have the following lemma.  First, fix notation by defining $D := D_{a,b}$ to be the dihedral group generated by $\tau_a$ and $\tau_b$ for axes $a, b \in X$.  Define $X_{a,b} = a^D \cup b^D$.  It is clear that $D_{a,b} = D_{b,a}$ and $X_{a,b} = X_{b,a}$.

\begin{lemma}\label{dihedralorbs}\textup{\cite[Lemma 3.2]{axialconstruction}}
Let $A$ be an axial algebra of Monster type, $a,b \in X$ and $D = D_{a,b}$.  Then we have the following:
\begin{enumerate}
\item[$1.$] $k := |a^D| = |b^D|$.
\item[$2.$] If $a$ and $b$ are in the same orbit under $D$, then $k = 1$, $3$, or $5$.
\item[$3.$] If $a$ and $b$ are in different orbits, then $k = 1$, $2$, or $3$.
\end{enumerate}
Moreover, the Norton-Sakuma algebra generated by $a$ and $b$ has type $n\textrm{L}$, where $n = |X_{a,b}|$.
\end{lemma}

\begin{proposition}\label{k-algebra}
Let $a, b$ be axes in an axial algebra of Monster type such that $\lla a, b \rra \cong n\L$.  Then, $(\tau_a \tau_b)^n = 1$.
\end{proposition}
\begin{proof}
By Lemma \ref{dihedralorbs}, the orbit of $a$ under $D = D_{a,b}$ has size $k$.  So, $z := (\tau_a \tau_b)^k$ fixes $a$.  Similarly, it fixes $b$ and so $z$ acts trivially on the subalgebra $\lla a, b\rra$.  Note that $\tau_a$ and $\tau_b$ both invert $z$.  Since $z$ fixes $a$ and $b$, it commutes with $\tau_a$ and $\tau_b$.  Hence, $z^{-1} = z$ and $z^2 = 1$.

If $n$ is even, then by Lemma \ref{dihedralorbs}, $n = 2k$ and we are done.  If $n$ is odd, then $n = k$, but additionally $a^D = b^D$ and so $\tau_a$ and $\tau_b$ are conjugate in $D$.  However, $\tau_a$ and $\tau_b$ generate the dihedral group $D$ and they can only be conjugate in $D$ if the order of $D$ is twice an odd number.  Hence, $z$ cannot be of order two and $1 = z=(\tau_a \tau_b)^n$.
\end{proof}

\begin{corollary}\label{6trans}
Every Miyamoto group of an axial algebra of Monster type is a $6$-transposition group.
\end{corollary}

As can be seen in the table, some Norton-Sakuma algebras are subalgebras of other Norton-Sakuma algebras.  Namely, we have the following inclusions, where $c,d \in X_{a,b}$.
\[
\begin{array}{c|c}
\lla a,b \rra & \lla c,d \rra \\
\hline
4\textrm{A} & 2\textrm{B} \\
4\textrm{B} & 2\textrm{A} \\
6\textrm{A} & 2\textrm{A} \\
6\textrm{A} & 3\textrm{A}
\end{array}
\]

\subsection{Shapes}\label{sec:shapes}

Suppose we have an algebra $(A, X)$ with a given Miyamoto group $G(X)$, a known action of $G := G(X)$ on $X$ and a known $\tau$-map.  Then there is still a possibility of multiple non-isomorphic algebras $A$ under this restrictive conditions.  This is because for two axes $a, b \in X$, the above information does not tell us which particular Norton-Sakuma subalgebra is generated by $a$ and $b$.  The shape formalises such choices.  The complete description of shapes is given in \cite{axialconstruction}, here we just want to give the reader some necessary details.

Every pair of axes $a \neq b$ gives us a Norton-Sakuma algebra $n\L$ where $n$ is known by Lemma \ref{dihedralorbs}.  All Norton-Sakuma algebras are symmetric, that is they admit an automorphism switching $a=a_0$ and $b=a_1$.  So we can think of $a,b$ as a $2$-element subset of $X$.  Note that conjugate subsets generate isomorphic Norton-Sakuma algebras.  So we may consider pairs up to the action of $G$.  Additionally, the little table at the end of Section \ref{sec:dihedral} indicates dependences between the choices of different pairs $a,b$ and $c,d$ caused by inclusion between the Norton-Sakuma algebras.  We denote by $X \choose 2$ the set of all $2$-element subsets of $X$.

\begin{definition}
The \emph{shape graph} $\Gamma$ has vertices given by the orbits of $G$ on $X \choose 2$ and an edge between $\{a,b\}^G$ and $\{c,d\}^G$ when $X_{a,b}$ is contained in $X_{c,d}$ or vice versa.
\end{definition}

Choosing a suitable Norton-Sakuma algebra $n\L$ for one vertex in a connected component of $\Gamma$ specifies the choices for all the other vertices in that component.  However, if the component contains a vertex with $n=5$, or $6$, then because these can only be $5\A$ and $6\A$, we don't have a choice for these components.  However, in our project no such components exists and so we have one binary choice for each connected component.  Making these choices for all components is choosing the shape of the algebra.  Where $G$ has a non-trivial outer automorphism group, this acts on the shapes and conjugate shapes give isomorphic algebras.  Hence we will only consider one shape in each orbit under $\Out(G)$.

Note that we describe the shape by a sequence of Norton-Sakuma algebra names which indicate our choices for the components, one name per component.  For example, for the group $G = 2^2$ on $6$ axes, considered in Proposition \ref{4B2B}, there are two $G$-orbits of $4\L$ subalgebras, but the subalgebras from the two orbits intersect in a common $2\L'$ subalgebra and hence lie in the same connected component.  So there is only one choice for $\L$ and we label the shape $4\B \, (2\B)^2$ rather than $(4\B)^2 \, (2\B)^2$ in Table \ref{tab:min0dims}.

The Partial Axial Algebras package \cite{ParAxlAlg} in {\sc magma} has functions to compute all possible $\tau$-maps for a given action of a group $G$ on $X$ and enumerate all possible shapes for a given $G$, $X$ and $\tau$-map.  The main part of the package implements the expansion algorithm in \cite{axialconstruction} that starts with a shape and either outputs the largest axial algebra of the given shape, which can be trivial (in which case we say the shape \emph{collapses}), or it stops when the next step would expand to an object which is too big for the current computers.

%%%%%%%%%%%%%%%%%%%%%%%%%%%%%%%%%%%%%%%%%%%%%%%%%%%%%%%%%%

\section{Groups for $3$-generated $4$-algebras}\label{sec:groups}

In this paper, we are interested in classifying all $3$-generated axial algebras of Monster type which do not have a $5\A$, or $6\A$ subalgebra - we call such algebras $4$-algebras.  By Proposition \ref{k-algebra}, the Miyamoto group of a $4$-algebra is a $4$-transposition group $(G, T)$ with respect to the normal set of Miyamoto involutions $T$.  Recall that a Miyamoto involution can be trivial, so we allow $1 \in T$.

Suppose that $A = \lla a, b,c \rra$ is a $3$-generated axial algebra of Monster type.  Then its Miyamoto group is generated by $x:=\tau_a$, $y:=\tau_b$ and $z:=\tau_c$.  So we first begin by classifying all $4$-transposition groups $(G, T)$ which are generated by three elements $x, y, z$ of $T$.  After this, we investigate all possible actions of $G$ on a putative set $X$ of axes together with a map $\tau \colon X \to G$ with image $T$.  Also $X$ must contain three elements $a$, $b$ and $c$ which $\tau$ maps to $x$, $y$ and $z$, such that $X = a^G \cup b^G \cup c^G$.   Since the Miyamoto group acts faithfully on $A$ and hence on $X$, we require our action to be faithful.  Interestingly, this condition eliminates many cases. 

Note that we do not assume that the orbits $a^G$, $b^G$ and $c^G$ are pairwise disjoint and hence we do not assume that the conjugacy classes of $x$, $y$ and $z$ are pairwise distinct either.

\subsection{$3$-generated $4$-transposition groups}

We approach the problem via presentations. Clearly, every $t\in T$ satisfies $t^2=1$. Furthermore, for $s,t\in T$, we have either $(st)^3=1$, or $(st)^4=1$.  Note that the last case includes the case where $(st)^2=1$.  Since we will later take all quotients to build a full list of groups, it suffices to just add relations of the form $(st)^3=1$, or $(st)^4=1$.

We begin by imposing the relations $(st)^3=1$, or $(st)^4=1$ for each distinct pair $\{s,t\} \subset \{x,y,z\}$.  This gives four main cases corresponding to the number of pairs of each order.  Since products of conjugates of the generators must also have order at most $4$, we may also add extra relations of the form $(st^g)^n=1$, with $n = 2,3,4$.  We continue adding relations of this form until the group is indeed a $4$-transposition group.

Let us illustrate this with the largest case where $(xy)^4 = (xz)^4 = (yz)^4 = 1$.  We look at the elements $xy^z$ and $xz^y$.  If one of them has order $3$, then the group obtained is finite and we can quickly finish the analysis.  So suppose that they are both order $4$.  In this case we look at the extra element $yz^x$.  Again, if this has order $3$, the group is finite and we can finish the analysis.  So assume that it has order $4$.  Finally we look at the element $xx^{yz}$ and find that the group is finite when this element has order $3$, or order $4$.  The largest group that we get here is when $(xx^{yz})^4 = 1$ and then the group has order $32, 768$.  The group is not a group of $4$-transpositions, so we still have to find extra relations.  Both $yy^{xz}$ and $zz^{xy}$ have order $8$ and by forcing them to have order $4$ we finally get a $4$-transposition group $Q$, which has order $8, 192$.  Clearly every quotient of $Q$ is a $3$-generated $4$-transposition group and vice versa it follows from our analysis that every $3$-generated $4$-transposition group that is a $2$-group is a quotient of $Q$.

The above analysis is easily done in, for example, {\sc magma} and we had to consider $14$ cases where we encountered $19$ groups, whose quotients constitute all $3$-generated $4$-transposition groups.  In particular, we see that all such groups are finite.
Some of these however can be discarded, for example all the $2$-groups are already quotients of $Q$.  In order for us to make more precise claims, we first need to discuss similarity.

\subsection{Similar groups}

Using {\sc magma}, we could find all quotients of the above $19$ groups and list them up to isomorphism mapping the $3$ marked generators of one to the $3$ marked generators of the other in any order.  However, this is not the equivalence we need.

Suppose that $x'$ is conjugate to $x = \tau_a$, $y'$ to $y$, and $z'$ to $z$. Then, by Lemma \ref{autaxis}, there is an axis $a'$ in the orbit $a^G$ such that $x'=\tau_{a'}$. Similarly, 
we find $b'\in b^G$ and $c'\in c^G$ such that $y'=\tau_{b'}$ and $z'=\tau_{c'}$. Suppose that $G' := \langle x', y', z' \rangle = G$.  Then the algebra $A' := \lla a', b', c' \rra$ is invariant under the action of $G = G'$ and so it contains $a,b,c$.  Since $X$ is a closed set of axes and $G = G'$, we have $A' = A$.

Recall that a multiset is a set where we allow repeated elements.  We make the following definition:

\begin{definition}
Let $G = \langle x,y,z \rangle$ and $G' = \langle x', y', z' \rangle$ be two $3$-generated groups.  We say $G$ and $G'$ are \emph{similar} if there is an isomorphism $\phi \colon G\to G'$ such that the multiset $\{ (x')^{G'}, (y')^{G'}, (z')^{G'} \}$ coincides with the multiset $\{ \phi(x)^{G'}, \phi(y)^{G'}, \phi(z)^{G'} \}$.
\end{definition}

By the above argument, similar groups will give isomorphic algebras.  Note that if $G$ and $G'$ are similar via $\phi$ then $G/N$ is similar to $G'/N'$, where $N \unlhd G$ and $N'=N^\phi$.  Hence, we may consider our list of $19$ groups up to similarity.  This, and discarding smaller $2$-groups, reduces our list to seven groups.  Using {\sc magma} to take all quotients of groups in this list up to similarity, we get the following.

\begin{proposition}\label{numbergroups}
There are exactly $55$ $3$-generated $4$-transposition groups up to similarity.
\end{proposition}

Recall that since we allow generators to be trivial, the trivial group for example is one of these.  The largest group is the $2$-group $Q$, the largest non-$2$-group has order $3,888$ and the only non-soluble groups are $PSL(2,7)$ and $2 \times PSL(2,7)$.

\section{Configuration of axes}\label{sec:axes}

We now consider what the possible configurations of axes are, given a putative Miyamoto group. More precisely, for $X$ to be a set of axes in an axial algebra with Miyamoto group $G = \la T \ra$ generated by a normal set of involutions $T$, $G$ must act faithfully on $X$ and there must exist a map $\tau \colon X \to T$ such that for all $d \in X$ and $g \in G$
\begin{enumerate}
\item $\tau_d \in G_d$
\item $\tau_{d^g} = \tau_d^g$
\end{enumerate}
where we write $G_d$ for the stabiliser of $d \in X$ in $G$.  We have the following easy result.

\begin{lemma}\label{taufix}
Let $A$ be an axial algebra of Monster type with Miyamoto group $G$ and $d$ an axis of $A$.  Then, $\tau_d \in Z(G_d)$.
\end{lemma}
\begin{proof}
Clearly $\tau_d \in G_d$ since $\gr(1) = +$.  Let $g \in G_d$.  Then, $\tau_d^g = \tau_{d^g} = \tau_d$ and so $\tau_d \in Z(G_d)$.
\end{proof}

By property $2$ above, it is clear that once we have identified the stabiliser $G_d$ of an axis $d \in X$, then the action of $G$ on $d^G$ is identified with the coset action on $G$ on $G_d$.  By Lemma \ref{taufix}, we have that $\la \tau_d \ra \leq G_d \leq C_G(\tau_d)$.

In the first subsection, we give several general results which give us better control over the stabiliser.  We use this in the next subsection to show how to limit the choice of stabiliser and we apply this in our case.  In the third subsection, we use faithfulness of the action to eliminate more cases.  Finally, in the last subsection, we describe how to find the different possible configurations of axes $X$ given the stabilisers of the axes and we apply this to our situation.

\subsection{Stabilisers of axes}

Note that the results in this section depend on the properties of the Norton-Sakuma algebras and so they are specific to axial algebras of Monster type.

\begin{lemma}\label{symmetricorbs}
Let $d,e$ be two axes.  If $\tau_e \in G_d$, then $\tau_d \in G_e$.
\end{lemma}
\begin{proof}
Let $D := \la \tau_d, \tau_e \ra$.  Since $\tau_d$ and $\tau_e$ both fix $d$, $|d^D| = 1$.  Since $\lla d,e \rra$ is a Norton-Sakuma algebra, by Lemma \ref{dihedralorbs}, we must have $|e^D| =1$ also and so $\tau_d \in G_e$.
\end{proof}

Now, let us see when $\tau_d$ is trivial.

\begin{lemma}\label{identity}
Let $A$ be an axial algebra of Monster type and $d$ an axis of $A$.  Then, $\tau_d = 1$ if and only if $d$ is fixed by the entire Miyamoto group $G$.
\end{lemma}
\begin{proof}
Let $e \in X$ be another axis of $A$.  Suppose that $\tau_d=1$.  Then, by Lemma \ref{symmetricorbs}, $\tau_e \in G_d$.  Since this is true for every axis $e$ and $G$ is generated by the Miyamoto involutions, $G$ fixes $d$.

Conversely, suppose that $G$ fixes $d$.  Then, for every axis $e \in X$, $\tau_e$ fixes $d$.  So, by Lemma \ref{symmetricorbs}, $\tau_d$ fixes every axis $e \in X$.  However, $G$ acts faithfully on the axes, so $\tau_d=1$.
\end{proof}

We now consider the case where a Miyamoto involution $u$ is not the identity.

\begin{definition}
A non-trivial Miyamoto involution $u$ is \emph{unique} if there exists a unique axis $d \in X$ such that $u=\tau_d$.  We say an axis $d \in X$ is \emph{unique} if $\tau_d$ is unique.
\end{definition}

It is easy to see that when $u$ is unique, the stabilizer $G_d$ coincides with the centraliser $C_G(u)$.  In particular, it is as large as it can be.  By property $2$ of $\tau$, $G_d = C_G(u)$ if and only if there is a natural $G$-invariant bijection between the $G$-orbit $u^G$ and the conjugacy class $d^G$.

Recall that in this paper, we are interested in the $4$-algebra case.  In particular, we have no $5\A$, or $6\A$ subalgebras.  We allow these cases in the following results for completeness.

\begin{lemma}\label{uniqueprop}
Let $A$ be an axial algebra of Monster type and $d \in X$.
\begin{enumerate}
\item If there exists $e \in X$ such that the order of $\tau_d \tau_e$ is $5$ then $\tau_d$ is unique.
\item If $A$ has no subalgebras of type $6\A$ and there exists $e \in X$ such that the order of $\tau_d \tau_e$ is $3$ then $\tau_d$ is unique.
\end{enumerate}
\end{lemma}
\begin{proof}
Suppose $f \in X$ such that $\tau_f = \tau_f = u$.  Then $D_{d,e} = D_{f,e} =: D$.  If the order of $\tau_d \tau_e$ is $5$, then $\lla d, e \rra$ and $\lla f, e \rra$ must both be $5\A$ algebras.  Since these only have one orbit of axes under the action of $D$, $d$, $e$ and $f$ all lie in the same orbit of $D$ and so $\lla d, e \rra = \lla f, e \rra$.  A simple computation in the $5\A$ algebra shows that all the five axes have distinct Miyamoto involutions.

Suppose that the order of $\tau_d \tau_e$ is $3$.  Since by assumption $A$ contains no $6\A$ subalgebras, $\lla d, e \rra$ and $\lla f, e \rra$ are algebras of type $3\A$, or $3\C$.  These only have one orbit of axes under $D$ so, as above, the subalgebras are equal and again the involutions are distinct.

In both cases, $f= d$ and so $\tau_d$ is unique.
\end{proof}

\begin{corollary} \label{O2unique}
If $A$ has no $6\A$ subalgebras and the Miyamoto involution $u$ does not lie in $O_2(G)$, then $u$ is unique.
\end{corollary}
\begin{proof}
Suppose that the order of the product $uv$ is even for all $v \in u^G$.  Then, $\la u, v\ra$ is a $2$-group for all $v \in u^G$ and, by Baer's Theorem, $u \in O_2(G)$, a contradiction.  So, there exists some $v \neq u$ in $u^G$ such that the order of $uv$ is odd.  Since these are Miyamoto involutions, this order is either $3$, or $5$.  As we assume there are no $6\A$ subalgebras, by Lemma \ref{uniqueprop}, $u$ is unique.
\end{proof}

What can be said about the stabiliser $G_d$ of an axis when it is not unique?  We introduce a property that is slightly weaker than uniqueness.

\begin{definition}
An axis $d \in X$ is \emph{strong} if $\tau_d \neq \tau_e$ for any $e \in d^G$, $e \neq d$.  We say an involution $u$ is \emph{strong} if there exists a strong axis $d$ with $\tau_d = u$.
\end{definition}

We stress that it is possible that $d$ is a strong axis, but there could be another axis $d'$ with $\tau_{d'} = \tau_d$ which is not strong.  Clearly if an axis $d$, or Miyamoto involution $u = \tau_d$ is unique, then it is strong.  Note also that both properties are preserved by conjugation.  We have the following easy lemma.

\begin{lemma}
Let $d$ be an axis with a non-trivial Miyamoto involution $u = \tau_d$.  Then, the following are equivalent.
\begin{enumerate}
\item $d$ is strong.
\item There is a natural $G$-invariant bijection between $d^G$ and $u^G$.
\item $G_d= C_G(u)$.
\end{enumerate}
\end{lemma}

What can be said about the stabiliser $G_d$ of an axis which is not unique, or strong?  As already noted, $\la \tau_d \ra \leq G_d \leq C_G(u)$.  The next lemma allows us to find additional elements of $G_d$.  Note that, if we have no $6\A$ subalgebras and $d$ is not unique, by Lemma \ref{uniqueprop}, the order of $uv$ is $1, 2, 4$ for all other Miyamoto involutions $v$.

\begin{lemma} \label{stronginv}
Let $u = \tau_d \neq 1$ and $v$ be another Miyamoto involution.
\begin{enumerate}
\item If the order of $uv$ is $2$ and $v$ is strong, then $v \in G_d$.
\item If the order of $uv$ is $4$, then $[u,v]=(uv)^2 \in G_d$.
\end{enumerate}
\end{lemma}
\begin{proof}
Let $D := \la u,v\ra$ and $e \in X$ such that $\tau_e = v$.  Consider the action of $D$ on $B := \lla d,e \rra$, noting that $D$ may not act on $B$ faithfully.

If the order of $uv$ is $2$, then $B$ is either a $2\L$, or a $4\L$ algebra, where $\L$ can be either $\A$, or $\B$.  If in addition $v$ is strong, then we can assume that $e$ is strong and so $e$ is the only axis in its orbit with the Miyamoto involution $v$.  Suppose that $B \cong 4\L$.  Then, $D$ has two orbits of length $2$ on the axes.  In particular, $u=\tau_d$ conjugates $e$ to the other axis $f$ in $e^D$.  However, $u$ and $v$ commute, so $\tau_e = \tau_f = v$, contradicting the assumption that $v$ is strong.  Hence, $B \cong 2\L$ and $v$ fixes $d$.

If the order of $uv$ is $4$, then $u$ and $v$ do not commute and so the $D$-orbit containing $d$ cannot be of length $1$.  So it is of length $2$ or $4$.  Again, looking at the list of Norton-Sakuma algebras, it can never be length $4$, so it must be length $2$.  Hence, $(uv)^2 \in G_d$ (in fact $B \cong 4\L$ and $\la(uv)^2\ra$ is the kernel of the action on $B$).
\end{proof}

\subsection{Limiting the choices of stabilisers}

Let $d \in \{a, b, c\}$ be an axis and $u = \tau_d$.  We use the results of the previous section to find a normal subgroup $H_u$ of $C_G(u)$ which is as large as possible such that $\la u \ra \leq H_u \leq G_d \leq C_G(u)$ giving a lower bound for any possible stabiliser $G_d$.  Note that for any $w \in T$, $w = v^g$ for some $v \in \{x,y,z\}$ and $g \in G$.  So we may define $H_w := H_u^{g^{-1}}$ (this is well-defined since $H_u$ is normal in $C_G(u)$).

\begin{enumerate}
\item If $u = 1$, then by Lemma \ref{identity}, $G_d = G$ and so $H_1 = G$.
\item If $u \notin O_2(G)$, then by Corollary \ref{O2unique} $u$ is unique and so $H_u = C_G(u)$.
\item Otherwise, we initially set $H_u = \la u, (uv)^2 : v \in T \ra$, using part $2$ of Lemma \ref{stronginv}.
\end{enumerate}

At this point, we have some $H_u$ for each $u \in \{x,y,z\}$ (and also their conjugates).  If $H_u = C_G(u)$, then $u$ and its entire class $u^G$ are strong.  Let $S$ be the union of the strong classes in $T$.  If $S = T$, then we are done and we know the stabilisers.  Otherwise, for those $H_u$ which do not equal $C_G(u)$, we iteratively grow $H_u$, using part $1$ of Lemma \ref{stronginv} and Lemma \ref{symmetricorbs}.

\begin{enumerate}
\item[$4.$] We add $S \cap C_G(u)$ to the generators of $H_u$.
\item[$5.$] For $w \in T \cap (C_G(u) - H_u)$, if $u \in H_w$, then we add $w$ to the generators of $H_u$.
\item[$6.$] If now $H_u = C_G(u)$, $u$ is strong.  We add $u^G$ to $S$.
\item[$7.$] If any of $H_x$, $H_y$, or $H_z$ have increased, then we iterate.
\end{enumerate}

By Proposition \ref{numbergroups}, there are $55$ $3$-generated $4$-transposition groups up to similarity.  Of these, thirteen have all three generators being either unique, or the identity, a further seven have two with this property, three have one with this property, while the remaining $32$ have no unique or identity generators.  At the end of the above procedure, $29$ groups have all three generators being strong, $11$ have two strong generators, $4$ have one and only ten groups have no strong generators.  Even in the cases where $H_u \lneqq C_G(u)$, $H_u$ is still quite big.  In fact, the largest index of $H_u$ in $C_G(u)$ is $8$, which occurs just once.

Let us also record the following easy fact which we will need later.

\begin{lemma}
$H_u \unlhd C_G(u)$.
\end{lemma}

This is because at every step we expand $H_u$ by a normal set of involutions.

\subsection{Faithfulness}

We will use the following observation to eliminate some of the cases as the Miyamoto group must act faithfully.

\begin{lemma}\label{faithful}
If $N \leq G_a \cap G_b \cap G_c$ is a normal subgroup of $G$, then $N = 1$.
\end{lemma}
\begin{proof}
For any axis $d \in X$, $d$ is conjugate to one of $a$, $b$ or $c$.  So $N$ is in the stabiliser $G_d$ and hence $N$ is in the kernel of the action.
\end{proof}

We will use this lemma later, however even now we can eliminate some cases without knowing the exact axis stabilisers.

\begin{corollary}
If $1 \neq N \leq H_x \cap H_y \cap H_z$ is a normal subgroup of $G$, then $G$ cannot be a Miyamoto group.
\end{corollary}

In fact this faithfulness condition eliminates 31 of the 55 groups, including the large $2$-groups.  This leaves $24$ groups (with marked generators $x$, $y$ and $z$ and known subgroups $H_x$, $H_y$ and $H_z$) whose orders are 1, 4, 4, 4, 6, 6, 8, 16, 18, 24, 24, 32, 64, 64, 72, 96, 96, 128, 168, 256, 512, 1024, 1152 and 1296.

\subsection{Configurations of axes}

For a given group $G$, we now want to list all possible triples of stabilisers.  We know that the stabiliser $G_d$ of an axis $d \in \{a,b,c\}$ satisfies $H_u \leq G_d \leq C_G(u)$, where $u = \tau_d$.  Recall that $H_u$ is normal in $C_G(u)$, hence $G_d$ corresponds to a subgroup $\bar G_d$ of the factor group $C_G(u)/H_u$.  Since the index $|H_u:C_G(u)| \leq 8$ in all cases, there are few possible options.  However, we may use the next lemma to limit the choices even further.

\begin{lemma}
Suppose $A = \lla a,b,c \rra$.  Let $a' \in a^{C_G(x)}$, $b' \in b^{C_G(y)}$ and $c' \in c^{C_G(z)}$.  Then, $A = \lla a', b' , c' \rra$.
\end{lemma}
\begin{proof}
Clearly $\tau_{a'} = \tau_a = x$ and similarly for $b'$ and $c'$.  Hence the Miyamoto group corresponding to $\{a', b', c'\}$ is $G$ and therefore $\lla a', b', c'\rra$ contains $a,b,c$ and so is equal to $A$.
\end{proof}

This means that we may choose the images $\bar G_d$ up to conjugation in the factor group and this is independent for each axis $a$, $b$ and $c$.  Also, if $u \neq 1$ is central in $G$ then we should discard the possibility of $G_d=G$ because of Lemma \ref{identity}.

For our $24$ remaining groups, we enumerate all possible stabilisers and find $257$ possible cases.  We use Lemma \ref{faithful} to discard $138$ of these, of those remaining a further $20$ are failed by Lemma \ref{identity} and finally Lemma \ref{symmetricorbs} allows us to discard another $56$.  This leaves $43$ cases in total.

The next step is to build the possible actions of $G$ on the set of axes.  For each stabiliser $G_d$, we can easily recover the action of $G$ on the orbit $d^G$; it is just isomorphic to the action of $G$ on the cosets of $G_d$.  However, we cannot assume that $a^G$, $b^G$ and $c^G$  are disjoint, so we must be careful building the possible sets of axes $X$.  If two axes, say $a$ and $b$, have Miyamoto involutions which are conjugate and the axis stabilisers $G_a$ and $G_b$ are also conjugate, then the orbits $a^G$ and $b^G$ have isomorphic $G$-actions. So, a priori, there are two possibilities to combine them.  Either $a^G \cup b^G$ is the disjoint union of $a^G$ and $b^G$, or $a^G \cup b^G = a^G = b^G$.  However, the first option is not valid when either $a$, or $b$ is unique by the definition of uniqueness (note that if either is unique, then they both are).

Define $\tau \colon X \to G$, by $\tau_a = x$, $\tau_b = y$ and $\tau_c = z$ and extend to $X$ using the action of $G$ on $X$.  By definition, we have $\tau_x \in G_x$ and $\tau_{x^g} = \tau_x^g$ for all $x \in X$ and $g \in G$.  So $\tau$ is well-defined.  We also have the following.

\begin{lemma}
Let $d,e \in X$ and define $D = \la \tau_d, \tau_e \ra$.  Then $k := |d^D| = |e^D|$ and moreover
\begin{enumerate}
\item If $d$ and $e$ are in the same orbit, then $k = 1$, or $3$.
\item If $d$ and $e$ are in different orbits, then $k = 1$, or $2$.
\end{enumerate}
\end{lemma}
\begin{proof}
Since $G$ is a $4$-transposition group, $n := |\tau_d\tau_e|$ is at most $4$ and so $D$ is a dihedral group of order $2n$.  Since $\tau_d \in G_d$, $|d^D|$ has size at most $n \leq 4$.  However, by part $2$ of Lemma \ref{stronginv}, if $|\tau_d\tau_e| = 4$, then $(\tau_d\tau_e)^2 \in G_d$.  So $D \cap G_d$ has order at least $4$ and so $d^D$ has order at most $2$.  In particular, in all cases $|d^D| \leq 3$.

Now suppose that $d^D = e^D$. If $d=e$, then $d^D = e^D = \{d\}$ and $k=1$.  So assume that $d \neq e$.  Then $|d^D|$ cannot have size $2$.  Indeed, suppose $d^D = \{ d, e \}$.  Then $\tau_d$ fixes $d$ and so must fix $e$ also.  Similarly for $\tau_e$ and hence $D$ acts trivially on $\{d,e\}$, a contradiction.

Finally, suppose that $d^D$ and $e^D$ are different orbits and so clearly $d \neq e$.  First assume that one of the orbits has length $3$, say $d^D$.  Then $|\tau_d\tau_e|$ must have order $3$ and so $\tau_d$ and $\tau_e$ are not in $O_2(G)$.  By Corollary \ref{O2unique} such axes must be unique if we do not have $6\A$ subalgebras.  Since we took this into account when building $X$, $\tau_d$ and $\tau_e$ are indeed unique.  However, since $\tau_{d^g} = \tau_d^g$ for $g \in D$ and $D=S_3$ has three involutions all in one orbit, there exists $f \in a^D$ such that $\tau_f = \tau_e$, contradicting the uniqueness of $\tau_e$.

It remains to consider the case where there are two distinct orbits of length at most $2$.  For a contradiction, suppose that $|d^D| = 1$ and $|e^D| = 2$.  So $\tau_e \in G_d$.  This situation is discussed in Lemma \ref{symmetricorbs} and since we discarded the cases of stabilisers which do not satisfy the lemma, we have $\tau_d \in G_e$.  So $|e^D|=1$, a contradiction.
\end{proof}

Comparing this to Lemma \ref{dihedralorbs}, we see that this is exactly what a $\tau$-map satisfies in axial algebras of Monster type.  Moreover, since the $5\A$ algebra is the only one with an orbit of length $5$ and the $6\A$ algebra is the only one with two disjoint orbits each of length $3$, we see that any algebra on $X$ with $\tau$-map $\tau$ could contain no $5\A$, or $6\A$ subalgebras.  Hence it would be a $4$-algebra.

Finally, we may remove from our list of possible actions of $G$ on $X$ those which are in fact $2$-generated as these give precisely the known Norton-Sakuma algebras.  After removing some isomorphic actions, we get a total of $31$ different cases to check.

%%%%%%%%%%%%%%%%%%%%%%%%%%%%%%%%%%%%%%%%%%%%%%%%%%%%

\section{Forbidden configurations of axes}\label{sec:forbidden}

Even though we have only $31$ actions to check, the number of possible shapes for each action can be very large (see Table \ref{tab:0dim}).  Indeed, the total number of possible shapes on our $31$ actions is $11, 257$ (found using routines in the magma package \cite{ParAxlAlg}).  This is too large a number to naively do each case in turn using the expansion algorithm.  Instead we use a different approach.

Suppose that $Y \subseteq X$ is a closed subset of axes; that is it is closed under the restricted Miyamoto group $G(Y)$.  Then any shape on $X$ restricts to a shape on $Y$.  If we know that the shape on $Y$ collapses, then the shape on $X$ must also collapse and we don't need to run the algorithm on this case.

So we use the following strategy: we find a short list of collapsing shapes on small closed sets $Y$.  Then for larger sets of axes $X$, we search for closed subsets $Y$ from this list and we discard those shapes on $X$ which restrict to collapsing shapes on $Y$.  The list we found has $18$ such collapsing shapes and, as we see below, this is astonishingly powerful in reducing the number of cases to consider.

In Table \ref{tab:min0dims} we list our $18$ forbidden shapes, which we show collapse using the expansion algorithm.  Each of them is $3$-generated (and so appears as one of our $11,257$ shapes) and they are all minimal with respect to being collapsing and being properly contained in some shape.  In Table \ref{tab:0dim} we list the number of times they appear as a subshape of a larger set of axes (including themselves).  In total, of the $11, 257$ cases we had to consider, $11, 134$ contain one of these $18$ minimal collapsing configurations, leaving just $123$ to compute individually using the expansion algorithm.  These are the only shapes which can lead to non-trivial algebras, although some of these may still collapse in the end.

\setlength{\tabcolsep}{8pt}
\begin{tabularx}{\textwidth}{cccY}
\hline
$G$ & axes & shape & number of shapes where it is a subshape\\
\hline
$2^2$ & 2+2+2 & $4\A \, 2\A \, 2\B$ & 9937\\
$2^2$ & 2+2+2 & $4\B \, (2\B)^2$ & 5604\\
\\
$S_3$ & 1+3 & $3\C\, 2\A$ & 27\\
\\
$2^3$ & 2+2+4 & $(4\A)^2 \, (2\A)^2 \, 2\B$ & 3701\\
$2^3$ & 2+2+4 & $4\A \, 4\B \, (2\A)^3$ & 796\\
$2^3$ & 2+2+4 & $(4\B)^2 \, (2\A)^3$ & 4432\\
$2^3$ & 2+2+4 & $(4\B)^2 \, 2\A \, 2\B \, 2\A$ & 6836\\
$2^3$ & 2+2+4 & $(4\B)^2 \, 2\A \, (2\B)^2$ & 6634\\
$2^3$ & 2+2+4 & $(4\B)^2 \, (2\B)^2 \, 2\A$ & 3701\\
\\
$2^3$ & 4+4+4 & $(4\B)^3 \, (2\A)^3$ & 1490\\
\\
$S_4$ & 3+6 & $4\A \, 3\C \, 2\A$ & 5\\
\\
$2^2 \wr 2$ & 4+4+4 & $(4\A)^3 \, 2\B \, 2\A$ & 2685\\
$2^2 \wr 2$ & 4+4+4 & $(4\A)^3 \, (2\B)^2$ & 2292\\
\\
$2 \wr 2^2$ & 4+4+8 & $4\B \, (4\A)^3 \, 2\A$ & 591\\
$2 \wr 2^2$ & 4+4+8 & $4\B \, (4\A)^2 \, 4\B \, 2\A$ & 676\\
$2 \wr 2^2$ & 4+4+8 & $4\B \, 4\A \, (4\B)^2 \, 2\A$ & 571\\
\\
$S_3 \wr 2$ & 6+6 & $4\A \, 3\A \, 3\C$ & 20\\
$S_3 \wr 2$ & 6+6 & $4\A \, (3\C)^2$ & 13\\
\hline
\\
\caption{Minimal collapsing shapes}\label{tab:min0dims}
\end{tabularx}

\subsection{Why do shapes collapse?}

What we observe is a very interesting phenomenon: almost all shapes collapse!  When axial algebras were first being computed for concrete groups, typically each shape led to a  non-trivial algebra.  However, it was first noticed in \cite{sanhan}, that many of the possible shapes collapse and that was the beginning of this project.

This phenomenon drastically changes our point of view and begs the question: why do shapes collapse?  This cannot be answered without understanding this problem theoretically.  That is by providing handmade proofs.

Here we give short proofs for two of the $18$ forbidden shapes.  The first proof involves the shape which appears second in Table \ref{tab:0dim}.

\begin{proposition}\label{4B2B}
Let $X := \{a,b,c,d,e,f\}$ be a putative set of six axes for the group $2^2$ with $\tau_x = (c,d)(e,f)$, for $x \in \{a,b\}$ and $\tau_x = (a,b)$ for $x \in \{c,d,e,f\}$.  If
\begin{enumerate}
\item $\lla a,c \rra \cong \lla a,e \rra \cong 4\B$ and
\item $\lla c, e \rra \cong \lla c,f \rra \cong 2\B$,
\end{enumerate}
then the shape collapses.
\end{proposition}
\begin{proof}
Consider a $4\B$ algebra $\la a_{-1}, a_0, a_1, a_2, v_{\rho^2} \ra$ as described in Table \ref{tab:sakuma}.  There are two orbits of axes and one see from the table that two axes in the same orbit generate a $2\A$ subalgebra whose third axis is $v_{\rho^2}$.  

Let $v_1$, respectively $v_2$, be the additional element $v_{\rho^2}$ of the $4\B$ algebra $\lla a,c \rra$, respectively $\lla a,e \rra$.  On the one hand, since both $a$ and $b$ are contained in the intersection $\lla a,c \rra \cap \lla a,e \rra$ and since $\lla a,b \rra$ is a $2\A$ subalgebra of each, we see that $v_1 = v_2$.  On the other hand, $v_2$ is the third axis in the other $2\A$ subalgebra $\lla e,f \rra$ contained in $\lla a, e \rra$.  Since $\lla c, e \rra$ and $\lla c, f \rra$ are $2\B$ subalgebras, $ce = cf = 0$.  However, from the fusion law given in Table \ref{tab:monsterfusion}, we see that $A_0(c)$ is a subalgebra. Hence, $cv_2 = 0$ also.  But $v_1$ is the third axis in the $2\A$-subalgebra $\lla c,d \rra$ and so $cv_1 \neq 0$, a contradiction as $v_1 = v_2$.
\end{proof}

As we see from the table, this case eliminates about half of all the cases where the Miyamoto group is a $2$-group.

The second forbidden shape that we provide a proof for is the third in the table and the proof arises as part of a more general algebra identification result.

\begin{proposition} \label{non-existence}
Let $A = \lla a,b,c \rra$ be a $3$-generated axial algebra of Monster type where $\lla a,b \rra \cong 2\L$, $\lla a,c \rra \cong 2\L'$, where $\L, \L', \in \{ \A, \B \}$, and $\lla b, c \rra$ is one of $3\A$, $3\C$, or $5\A$.  Then,
\begin{enumerate}
\item[$1.$] $L = L'$.
\item[$2.$] If $\lla a,b \rra \cong \lla a,c \rra \cong 2\A$, then $\lla b, c \rra \cong 3\A$ and $A \cong 6\A$.
\item[$3.$] If $\lla a,b \rra \cong \lla a,c \rra \cong 2\B$, then $A \cong \lla a \rra \oplus \lla b, c \rra$.
\end{enumerate}
\end{proposition}
\begin{proof}
First, note that the Miyamoto involution $\tau_a$ fixes $a$, $b$ and $c$, and hence it is trivial. Therefore, the Miyamoto group of $A$ is $G = \la \tau_b,\tau_c \ra$, which is isomorphic to either $S_3$ or $D_{10}$. In both cases, the group conjugates $b$ to $c$ while fixing $a$.  Hence, $\lla a,b \rra \cong \lla a,c \rra$ and so $L = L'$.

If $2\L=2\B$ then, since $A_0(a)$ is a subalgebra, we see that $a\lla b,c\rra = 0$ which means that $A$ is isomorphic to the direct sum $1\A \oplus 3\A$, $1\A \oplus 3\C$ or $1\A \oplus 5\A$.

Suppose that $2\L = 2\A$.  We may pick a basis $\{a,b,b'\}$ of $B=\lla a,b\rra$, where $b'$ is the extra axis in $B$.  Since $\tau_a$ acts trivially on $A$, $A_{\frac{1}{32}}$ is trivial.  So, with respect to $a$, $A$ is still $\mathbb{Z}_2$-graded but with the grading $A_+(a) = A_1(a)\oplus A_0(a)$ and $A_-(a) = A_{\frac{1}{4}}(a)$.  We denote the associated involution that negates $A_-(a)$ by $\sigma_a$ to distinguish it from $\tau_a$.  By a calculation in the $2\A$ algebra $B$, $b'=b^{\sg_a}$.  Hence $b'$ is also an axis in $A$.

We claim that $C=\lla b',c\rra = A$. Note that $\tau_b$ fixes $a$ and so, by Lemma \ref{autaxis}, $\sg_a^{\tau_b}=\sg_{a^{\tau_b}}=\sg_a$ and $\tau_b$ and $\sg_a$ commute.  However, $\tau_{b'}=\tau_{b^{\sg_a}}=\tau_b^{\sg_a}=\tau_b$ and hence $\la\tau_{b'},\tau_c\ra=\la\tau_b,\tau_c\ra=G$.  By Lemma \ref{dihedralorbs}, all the axes of $\lla b, c \rra$ under $G$ are in one orbit.  Hence, $C$, which is invariant under $\la\tau_{b'},\tau_c\ra = G$ also contains $b$.  Furthermore, since $C$ contains $b$ and $b'$, it must contain the whole of $B=\lla b,b'\rra$, and so it contains $a$.  Therefore, $C=A$, as claimed.

In particular, $A$ is generated by two axes, and so it must be one of the Norton-Sakuma algebras. Since $G=\la\tau_{b'},\tau_c\ra$ does not conjugate $b'$ to $c$, the algebra $A$ can only be of type $6\A$. It is well known that this algebra contains the algebra $3\A$ and not $3\C$ or $5\A$. Hence $\lla b,c\rra$ is of type $3\A$.
\end{proof}

\begin{corollary}\label{no3C2A}
There exists no $3$-generated axial algebra of Monster type of shape $3\C2\A$, or $5\A2\A$.
\end{corollary}

The first of these two shapes is the third in our table and it eliminates roughly a quarter of the cases where the group is not a $2$-group.

Clearly proving more results of this kind would be very interesting as would formulating conditions on the shapes that eliminate a majority of collapsing shapes without excluding shapes which lead to good algebras.

%%%%%%%%%%%%%%%%%%%%%%%%%%%%%%%%%%%%%%%%%%%%%%%%%%%%

\section{The remaining cases}\label{sec:results}

From Section \ref{sec:forbidden}, we know that there are only $123$ shapes left to consider out of a total of $11,257$.  For these, we use the expansion algorithm to try to build an algebra.  For $22$ shapes, the algorithm succeeded by collapsing the algebra.  We did not include these in the earlier table, since they do not glue properly into any other shape.  For the remaining $101$ shapes, we display a summary of the results in Table \ref{tab:0dim} and we show the shapes case by case in Table \ref{tab:results}.  We find $45$ non-trivial algebras and have $56$ cases left which the algorithm couldn't complete.

\begin{table}[!htb]
\renewcommand{\arraystretch}{1.1}
\centering
\begin{tabular}{ccC{1.8 cm}C{1.8 cm}C{1.8 cm}C{1.9 cm}}
\hline
Group & Axes & Number of Shapes & Collapsing shapes & Shapes giving a non-trivial algebra & Incomplete shapes\\
\hline
$1$ & $1+1+1$ & 4 & 0 & 4 & \\
$2^2$ & $1+2+2$ & 6 & 0 & 6 & \\
$2^2$ & $2+2+2$ & 2 & 0 & 1 & 1\\
$2^2$ & $2+2+2$ & 6 & 2 & 3 & 1\\
$S_3$ & $1+3$ & 4 & 1 & 3 & \\
$2^3$ & $2+2+4$ & 4 & 4 & 0 & \\
$2^3$ & $2+2+4$ & 18 & 12 & 4 & 2\\
$2^3$ & $2+4+4$ & 12 & 11 & 1 & \\
$2^3$ & $4+4+4$ & 20 & 16 & 1 & 3\\
$2 \times D_8$ & $4+4+8$ & 80 & 80 & 0 & \\
$3^2:2$ & $9$ & 5 & 1 & 2 & 2 \\
$S_4$ & $6$ & 4 & 0 & 4 & \\
$S_4$ & $3+6$ & 8 & 1 & 7 & \\
$2^2\wr 2$ & $4+4+4$ & 24 & 18 & 0 & 6 \\
$2^2\wr 2$ & $4+8+8$ & 216 & 214 & 0 & 2 \\
$2\wr 2^2$ & $4+4+8$ & 24 & 21 & 0 & 3 \\
$2\wr 2^2$ & $8+8+8$ & 288 & 284 & 0 & 4 \\
$2.4:D_8$ & $8+8+8$ & 364 & 357 & 0 & 7\\
$S_3 \wr 2$ & $6+6$ & 6 & 4 & 0 & 2 \\
$4^2:S_3$ & $12$ & 4 & 1 & 2 & 1 \\
$2^2:S_4$ & $6+12$ & 16 & 12 & 1 & 3 \\
$2^2:S_4$ & $12+12$ & 16 & 13 & 3 &  \\
$2^4.2^3$ & $8+8+16$ & 1560 & 1558 & 0 & 2 \\
$PSL(2,7)$ & $21$ & 4 & 0 & 3 & 1 \\
$2^5.2^3$ & $8+16+16$ & 2520 & 2514 & 0 & 6 \\
$2^5.2^3$ & $16+16+16$ & 1540 & 1535 & 0 & 5 \\
$2^5.D_8.2$ & $16+16+32$ & 2520 & 2520 & 0 &  \\
$2^4.2^4.2^2$ & $16+32+32$ & 1560 & 1560 & 0 &  \\
$2^4.2^4.2^2$ & $32+32+32$ & 364 & 363 & 0 & 1 \\
$A_4^2.D_8$ & $12+24$ & 32 & 30 & 0 & 2 \\
$3^4.D_8:2$ & $18+18+18$ & 26 & 24 & 0 & 2 \\
\hline
\end{tabular}
\caption{Summary}\label{tab:0dim}
\end{table}

There could be several reasons for the incomplete cases.  The expansion algorithm, if successful, builds the universal cover of the algebra for that shape over $\mathbb{Q}$.  However it will only complete if this is finite dimensional.  For some of the shapes this may not be the case.  Indeed, in \cite{maddyinfinite} Whybrow shows that the case of $G = 2^2$ acting on $2+2+2$ axes with shape $4\A$ completes to a $12$-dimensional algebra over the function field $\mathbb{Q}(t)$.  Each specialisation of $t$ to a value in $\mathbb{Q}$ gives a non-isomorphic algebra, so the cover over $\mathbb{Q}$ is infinite dimensional.  The second author together with Peacock has found different behaviour with $G = 2^2$ acting on $2+2+2$ axes with shape $4\A (2\A)^2$.  They used the expansion algorithm together with some additional code to show that, over a function field $\mathbb{Q}(t)$, the algebra collapses for all values of $t$ except $t=\frac{1}{128}$, when there is a $10$-dimensional algebra, and $t= -\frac{1}{128}$, where the algorithm doesn't complete.  So the universal cover is potentially also infinite dimensional in this case. This shows that even for very small groups, there can be quite complicated behaviour.

For some of the cases, the algorithm, if run for long enough, might complete but the magma implementation runs out of memory.  Note that this doesn't necessarily depend on the size of the group, it depends more on the number of expansions needed.  For example, for the Miyamoto group $G = 2^2$ acting on $2+2+2$ axes of shape $4\A (2\A)^2$, we ended up with a partial algebra of dimensional $24, 749$ and would need to expand from there to about $300$ million dimensions in order to find more relations, which is currently not practical. We hope that subsequent theoretical improvements to the expansion algorithm \cite{axialconstruction} will lead to more cases being completed.

We describe all the non-trivial algebras and cases where the algorithm could not complete in Table \ref{tab:results}.  The columns in the table are
\begin{itemize}
\item Miyamoto group.
\item Axes, where we give the size decomposed into the sum of orbit lengths.
\item Shape.
\item Dimension of the algebra.  A question mark indicates that our algorithm did not complete.
\item The minimal $m$ for which $A$ is $m$-closed.  Recall that an axial algebra is $m$-closed if it is spanned by products of length at most $m$ in the axes.
\item Whether the algebra has a Frobenius form that is non-zero on the set of axes $X$.  If it is additionally positive definite or positive semi-definite, we mark this with a pos, or semi, respectively.
\end{itemize}

We now comment on some of the results in the table.  All the non-trivial algebras constructed admit a Frobenius form.  This supports a conjecture in \cite{axialconstruction}.  Furthermore, we note that all the forms are positive semi-definite, with the vast majority being positive definite.  In the two cases where the form is semi-definite, the radical of the form is an ideal \cite{axialstructure}, so we may quotient by this ideal to get another axial algebra which has a positive definite Frobenius form.  In both cases, the radical of the form is $3$-dimensional and hence there is an $11$- and $13$-dimensional axial algebra, respectively.  Note that all Norton-Sakuma algebras except for $2\B$ are simple \cite{axialstructure}.  Hence, as no axes are contained in the radical, the two new quotients have the same shape as before.  Also, all the non-trivial algebras constructed are primitive, except for the trivial group of shape $(2\A)^3$, where the universal cover would have the $1$-eigenspace being at least $3$-dimensional.

The table contains one shape that the computer wasn't able to finish, the first line of the table, but we include it there as we have a handmade proof of the claim.  For the definition of a Matsuo algebra for a given $3$-transposition group, see \cite{Axial2}.

\begin{proposition}\label{2A^3}
Let $A$ be a $3$-generated primitive axial algebra of Monster type with trivial Miyamoto group $G$ and shape $(2\A)^3$.  Then, $A$ is in fact an axial algebra of Jordan type $\frac{1}{4}$.  Moreover, $A$ is isomorphic to one of the following:
\begin{enumerate}
\item[$1.$] The Matsuo algebra of dimension $3$ for the group $S_3$,
\item[$2.$] The Matsuo algebra of dimension $6$ for the group $S_4$,
\item[$3.$] The Matsuo algebra of dimension $9$ for the group $3:S_3$.
\end{enumerate}
\end{proposition}
\begin{proof}
Since the Miyamoto group $G$ (with respect to the Monster fusion law) is the trivial group, each axis has trivial $\frac{1}{32}$-eigenspace.  So we may restrict the fusion law to the Jordan fusion law of type $\frac{1}{4}$ (this is the law obtained from the Monster fusion law by removing the last column and last row) and hence $A$ is an axial algebra of Jordan type $\frac{1}{4}$.  By \cite[Theorem 5.4 (b)]{Axial2}, such an algebra is a (quotient of a) Matsuo algebra defined in terms of a $3$-generated $3$-transposition group.  These are all known and thus we obtain the list above.
\end{proof}

We note that the first case in the above lemma is in fact the Norton-Sakuma algebra $2\A$ and it is generated by any two of the three generators.  Hence we don't include this algebra in the table.

\setlength{\tabcolsep}{8pt}
\begin{longtable}{cccccc}
\hline
$G$ & axes & shape & dim & $m$ & form \\
\hline

$1$ & 1+1+1 & $(2\A)^3$ & 6,9 & 2,3 & pos \\
$1$ & 1+1+1 & $(2\A)^2 \, 2\B$ & 6 & 3 & pos \\
$1$ & 1+1+1 & $2\A \, (2\B)^2$ & 4 & 2 & pos \\
$1$ & 1+1+1 & $(2\B)^3$ & 3 & 1 & pos \\
\\
$2^2$ & 1+2+2 & $4\A \, (2\A)^2$ & 14 & 3 & semi \\
$2^2$ & 1+2+2 & $4\A \, 2\A \, 2\B$ & 10 & 3 & pos \\
$2^2$ & 1+2+2 & $4\A \, (2\B)^2$ & 6 & 2 & pos \\
$2^2$ & 1+2+2 & $4\B \, (2\A)^2$ & 5 & 1 & pos \\
$2^2$ & 1+2+2 & $4\B \, 2\A \, 2\B$ & 8 & 2 & pos \\
$2^2$ & 1+2+2 & $4\B \, (2\B)^2$ & 6 & 2 & pos \\
\\
$2^2$ & 2+2+2 & $4\A$ & ? &  & \\
$2^2$ & 2+2+2 & $4\B$ & 7 & 2 & pos \\
\\
$2^2$ & 2+2+2 & $4\A \, (2\A)^2$ & ? &  & \\
%$2^2$ & 2+2+2 & $4\A \, 2\A \,  2\B$ & 0 & 0 & - \\
$2^2$ & 2+2+2 & $4\A \, (2\B)^2$ & 9 & 3 & pos \\
$2^2$ & 2+2+2 & $4\B \, (2\A)^2$ & 11 & 2 & pos \\
$2^2$ & 2+2+2 & $4\B \, 2\A \, 2\B$ & 8 & 2 & pos \\
%$2^2$ & 2+2+2 & $4\B \, (2\B)^2$ & 0 & 0 & - \\
\\
$S_3$ & 1+3 & $3\A \, 2\A$ & 8 & 2 & pos \\
$S_3$ & 1+3 & $3\A \, 2\B$ & 5 & 2 & pos \\
%$S_3$ & 1+3 & $3\C \, 2\A$ & 0 & 0 & - \\
$S_3$ & 1+3 & $3\C \, 2\B$ & 4 & 1 & pos \\
\\
$2^3$ & 2+2+4 & $(4\A)^2  \, 2\A  \, (2\B)^2$ & ?  &  & \\
$2^3$ & 2+2+4 & $(4\A)^2  \, (2\B)^3$ & 13 & 3 & pos \\
$2^3$ & 2+2+4 & $4\A \, 4\B \,  2\A \, 2\B \,  2\A$ & 15 & 3 & pos \\
$2^3$ & 2+2+4 & $4\A \, 4\B \,  (2\B)^2 \, 2\A$ & 12 & 2 & pos \\
$2^3$ & 2+2+4 & $(4\B)^2  \, (2\A)^2 \, 2\B$ & ?  &  & \\
$2^3$ & 2+2+4 & $(4\B)^2  \, (2\B)^3$ & 10 & 2 & pos \\
\\
$2^3$ & 2+4+4 & $4\B \, 4\A \, 2\A$ & 16 & 2 & semi\\
\\
$2^3$ & 4+4+4 & $(4\A)^3 \, (2\B)^3$ & ?  &  & \\
$2^3$ & 4+4+4 & $(4\A)^2 \, 4\B \, (2\A)^2 \, 2\B$ & ?  &  & \\
$2^3$ & 4+4+4 & $4\A \, (4\B)^2  \, 2\B \, (2\A)^2$ & ?  &  & \\
$2^3$ & 4+4+4 & $(4\B)^3 \, (2\B)^3$ & 15 & 2 & pos \\
\\
$3^2:2$ & 9 & $(3\A)^4$ & ? &  & \\
%$3^2:2$ & 9 & $(3\A)^3 \, 3\C$ & 0 & 0 & - \\
$3^2:2$ & 9 & $(3\A)^2 \, (3\C)^2$ & ? &  & \\
$3^2:2$ & 9 & $3\A \, (3\C)^3$ & 12 & 2 & pos \\
$3^2:2$ & 9 & $(3\C)^4$ & 9 & 1 & pos \\
\\
$S_4$ & 6 & $3\A \, 2\A$ & 13 & 2  & pos \\
$S_4$ & 6 & $3\A \, 2\B$ & 13 & 3 & pos \\
$S_4$ & 6 & $3\C \, 2\A$ & 9 & 2 & pos \\
$S_4$ & 6 & $3\C \, 2\B$ & 6 & 1 & pos \\
\\
$S_4$ & 3+6 & $4\A \, 3\A \, 2\A$ & 23 & 3 & pos \\
$S_4$ & 3+6 & $4\A \, 3\A \, 2\B$ & 25 & 3 & pos \\
%$S_4$ & 3+6 & $4\A \, 3\C \, 2\A$ & 0 & 0 & - \\
$S_4$ & 3+6 & $4\A \, 3\C \, 2\B$ & 12 & 2 & pos \\
$S_4$ & 3+6 & $4\B \, 3\A \, 2\A$ & 13 & 2 & pos \\
$S_4$ & 3+6 & $4\B \, 3\A \, 2\B$ & 16 & 2 & pos \\
$S_4$ & 3+6 & $4\B \, 3\C \, 2\A$ & 9 & 1 & pos \\
$S_4$ & 3+6 & $4\B \, 3\C \, 2\B$ & 12 & 2 & pos \\
\\
$2^2\wr 2$ & $4+4+4$ & $(4\A)^3 \, (2\A)^2$ & ? & & \\
$2^2\wr 2$ & $4+4+4$ & $(4\A)^3 \, 2\A \, 2\B$ & ? & & \\
$2^2\wr 2$ & $4+4+4$ & $(4\A)^2 \, 4\B \, 2\B \, 2\A$ & ? & & \\
$2^2\wr 2$ & $4+4+4$ & $(4\A)^2 \, 4\B \, (2\B)^2$ & ? & & \\
$2^2\wr 2$ & $4+4+4$ & $4\A \, (4\B)^2 \, (2\A)^2$ & ? & & \\
$2^2\wr 2$ & $4+4+4$ & $4\A \, (4\B)^2 \, 2\A \, 2\B$ & ? & & \\
\\
$2^2\wr 2$ & $4+8+8$ & $(4\A)^4 \, 2\A \, (2\B)^4$ & ? & & \\
$2^2\wr 2$ & $4+8+8$ & $(4\A)^4 \, (2\B)^5$ & ? & & \\
\\
$2\wr 2^2$ & $4+4+8$ & $4\B \, (4\A)^3 \, 2\B$ & ? & &\\
$2\wr 2^2$ & $4+4+8$ & $4\B \, (4\A)^2 \, 4\B \, 2\B$ & ? & & \\
$2\wr 2^2$ & $4+4+8$ & $4\B \, 4\A \, (4\B)^2 \, 2\B$ & ? & & \\
\\
$2\wr 2^2$ & $8+8+8$ & $(4\A)^4 \, 4\B \, (2\A)^2 \, (2\B)^2$ & ? & & \\ % swapped the ordering of the columns here
$2\wr 2^2$ & $8+8+8$ & $4\B \, (4\A)^4 \, (2\A)^4$ & ? & & \\
$2\wr 2^2$ & $8+8+8$ & $(4\B)^2 \, (4\A)^3 \, (2\A)^3 \, 2\B$ & ? & & \\
$2\wr 2^2$ & $8+8+8$ & $(4\B)^3 \, (4\A)^2 \, (2\A)^2 \, (2\B)^2$ & ? & & \\
\\
$2.4:D_8$ & $8+8+8$ & $(4\A)^6 \, (2\A)^3 \, (2\B)^3$ & ? & & \\% swapped the ordering of the columns here
$2.4:D_8$ & $8+8+8$ & $(4\A)^6 \, (2\A)^2 \, (2\B)^4$ & ? & & \\
$2.4:D_8$ & $8+8+8$ & $(4\A)^6 \, 2\A \, (2\B)^5$ & ? & & \\
$2.4:D_8$ & $8+8+8$ & $(4\A)^6 \, (2\B)^6$ & ? & & \\
$2.4:D_8$ & $8+8+8$ & $(4\A)^4 \, (4\B)^2 \, (2\B)^2 \, 2\A \, (2\B)^3$ & ? & & \\
$2.4:D_8$ & $8+8+8$ & $(4\A)^2 \, (4\B)^4 \, 2\A \, (2\B)^5$ & ? & & \\
$2.4:D_8$ & $8+8+8$ & $(4\B)^6 \, (2\A)^3 \, (2\B)^3$ & ? & & \\
\\
$S_3\wr 2$ & 6+6 & $4\A \, (3\A)^2$ & ?  &  & \\
$S_3\wr2$ & 6+6 & $4\B \, (3\A)^2$ & ?  &  & \\
\\
$4^2:S_3$ & 12 & $4\A \, 3\A$ & ?  &  & \\
$4^2:S_3$ & 12 & $4\A \, 3\C$ & 15 & 2  & pos \\
%$4^2:S_3$ & 12 & $4\B \, 3\A$ & 0 & 0  & -\\
$4^2:S_3$ & 12 & $4\B \, 3\C$ & 15 & 2 & pos \\
\\
$2^2:S_4$ & 6+12 & $(4\A)^3 \, 3\A \, 2\A$ & ?  &  & \\
$2^2:S_4$ & 6+12 & $(4\A)^3 \, 3\C \, 2\A$ & ?  &  & \\
$2^2:S_4$ & 6+12 & $4\A \, (4\B)^2 \, 3\A \, 2\B$ & ?  &  & \\
$2^2:S_4$ & 6+12 & $4\A \, (4\B)^2 \, 3\C \, 2\B$ & 30 & 2 & pos \\
\\
$2^2:S_4$ & 12+12 & $(4\A)^3 \, 4\B \, 3\C$ & 60 & 3 & pos \\
$2^2:S_4$ & 12+12 & $4\A \, (4\B)^3 \, 3\A$ & 59 & 3 &  pos \\
$2^2:S_4$ & 12+12 & $4\A \, (4\B)^3 \, 3\C$ & 42 & 2 &  pos \\
\\
$2^4.2^3$ & $8+8+16$ & $(4\A)^5 \, (4\B)^2 \, (2\A)^2 \, (2\B)^4$ & ?  &  & \\ % swapped the ordering of the columns here
$2^4.2^3$ & $8+8+16$ & $4\B \, (4\A)^6 \, (2\B)^2 \, (2\A)^2 \, (2\B)^2$ & ?  &  & \\
\\
$PSL(2,7)$ & 21 & $4\A \, 3\A$ & ? &  & \\
$PSL(2,7)$ & 21 & $4\A \, 3\C$ & 57 & 3 & pos \\
$PSL(2,7)$ & 21 & $4\B \, 3\A$ & 49 & 2 & pos \\
$PSL(2,7)$ & 21 & $4\B \, 3\C$ & 21 & 1 & pos \\
\\
$2^5.2^3$ & $8+16+16$ & $(4\A)^7 \, 4\B \, (2\A)^2 \, (2\B)^4$ & ? &  & \\ % swapped the ordering of the columns here
$2^5.2^3$ & $8+16+16$ & $(4\A)^7 \, 4\B \, 2\A \, (2\B)^5$ & ? &  & \\
$2^5.2^3$ & $8+16+16$ & $(4\A)^4 \, (4\B)^2  \, (4\A)^2\, (2\B)^2 \, (2\A)^4$ & ? &  & \\
$2^5.2^3$ & $8+16+16$ & $(4\A)^4 \, (4\B)^4 \, (2\B)^2 \, (2\A)^4$ & ? &  & \\
$2^5.2^3$ & $8+16+16$ & $(4\A)^3  \, (4\B)^2 \, (4\A)^2\, 4\B \, 2\A \, (2\B)^5$ & ? &  & \\
$2^5.2^3$ & $8+16+16$ & $(4\A)^2 \, (4\B)^4 \, 4\A \, 4\B \, (2\A)^2 \, (2\B)^4$ & ? &  & \\
\\
$2^5.2^3$ & $16+16+16$ & $(4\A)^9 \, (2\A)^6$ & ? &  & \\
$2^5.2^3$ & $16+16+16$ & $(4\A)^7 \, (4\B)^2 \, (2\A)^6$ & ? &  & \\
$2^5.2^3$ & $16+16+16$ & $(4\A)^5 \, (4\B)^2 \, 4\A \, 4\B \, (2\B)^6$ & ? &  & \\
$2^5.2^3$ & $16+16+16$ & $(4\A)^4 \, 4\B \, 4\A \, (4\B)^3 \, (2\A)^6$ & ? &  & \\
$2^5.2^3$ & $16+16+16$ & $4\B \, (4\A)^3 \, (4\B)^5 \, (2\A)^6$ & ? &  & \\
\\
$2^4.2^4.2^2$ & $32+32+32$ & $(4\A)^3 \, (4\B)^3 \, (4\A)^6$ & ? &  & \\
\\
$A_4^2.D_8$ & $12+24$ & $(4\A)^2 \, (3\A)^2 \, 2\A$ & ? &  & \\
$A_4^2.D_8$ & $12+24$ & $4\A \, 4\B \, (3\A)^2 \, 2\A$ & ? &  & \\
\\
$3^4.D_8:2$ & $18+18+18$ & $4\A \, (3\A)^6$ & ? &  & \\
$3^4.D_8:2$ & $18+18+18$ & $4\B \, (3\A)^6$ & ? &  & \\
\hline
\caption{$3$-generated $4$-algebras}\label{tab:results}
\end{longtable}

\end{document}